\documentclass[a4paper,12pt]{article}

\usepackage[T1]{fontenc}
\usepackage{mathptmx}
\usepackage{amsthm,wrapfig,amsmath,amssymb,amsfonts,setspace,verbatim,graphicx,mathtools,mathrsfs,commath,float,mdframed,frame,xcolor,qtree,enumitem,ulem,xfrac,afterpage}
\usepackage{multirow}
\usepackage[hidelinks]{hyperref}
\usepackage[T1]{fontenc}

\usepackage[margin=1.2in]{geometry}

\newtheorem{ut}{Theorem}
\newtheorem{up}[ut]{Proposition}
\newtheorem{ul}[ut]{Lemma}

\newtheorem{uc}[ut]{Corollary}
\newtheorem{ucl}[ut]{Claim}

\theoremstyle{remark}

\theoremstyle{definition}

\title{On a local property of fences and fans}

\author{David S. Lipham}

\date{}

\begin{document}

\maketitle

\begin{abstract}
Two closely related classes of topological spaces are fences and fans. A \textit{fence} is a compact metric space whose  connected components are  arcs and singletons. A \textit{fan} is a continuum formed by joining arcs at a common vertex, in such a way that intersections of subcontinua are always connected. We prove that every fence can be embedded in the plane and that both fences and fans admit a basis consisting of pierced open sets. This resolves a question  by Iztok Banič, Goran Erceg, Ivan Jelić, Judy Kennedy, and Van Nall.
\end{abstract}

\

\

\

\section{Introduction}

A \textbf{continuum} is a non-empty compact connected metric space.  

Following \cite{pans}, we define a continuum $X$ to be a \textbf{fan} if:
\begin{itemize}
\item[(1)] there exists $v\in X$  
 and a collection of arcs $\mathcal L$ such that $X=\bigcup \mathcal L$, and $L\cap L'=\{v\}$ for every two elements $L,L'\in\mathcal L$; and 
\item[(2)] for every two subcontinua $H$ and $K$ of $X$, the intersection  $H\cap K$ is connected.
\end{itemize}
If $X$ is a fan then any point $v$ of $X$ which satisfies (1) is necessarily unique and is called the \textbf{vertex} of $X$. Each fan is uniquely and hereditarily arcwise connected. Additionally, fans are known to be  $1$-dimensional. It is an open question as to whether every $1$-dimensional continuum that satisfies (1)  is a fan \cite[Problem 3.4]{pans}. The authors of \cite{pans} show that under certain conditions, one of which is the Jure property defined below, the question has a positive answer (cf. \cite[Corollary 4.29]{pans}).

A \textbf{fence} is a compact  metric space whose components are arcs and points \cite{bas}. It is easy to prove that each subcontinuum of a fan which misses the  vertex is either an arc or a singleton. Thus,

\begin{up}If $X$ is a fan and $V$ is any open set containing the vertex of $X$, then $X\setminus V$ is a fence.\end{up}

Here we are interested in a local property  having  to do with arcs that cross, as opposed to bend at,  neighborhood boundaries. Suppose that $X$ is a fence or a fan, and $U\subset X$ is open.   A point $x\in \partial U$  is a \textbf{piercing point} for $U$ if for each open set $V$  which contains $x$, and every arc $\alpha\subset X$ which goes through $x$ (so $x$ is not an endpoint of $\alpha$), $\alpha\cap V$ intersects both $U$ and $X\setminus \overline U$. This is easily seen to be equivalent to the definition in \cite{pans}. 
An   open set $U\subset X$  is \textbf{pierced} if every point of $\partial U$ is piercing.    Note that the boundary of any pierced open set is 0-dimensional, as every compact $1$-dimensional subset of $X$ contains an arc. 

\begin{figure}\centering\includegraphics[scale=0.245]{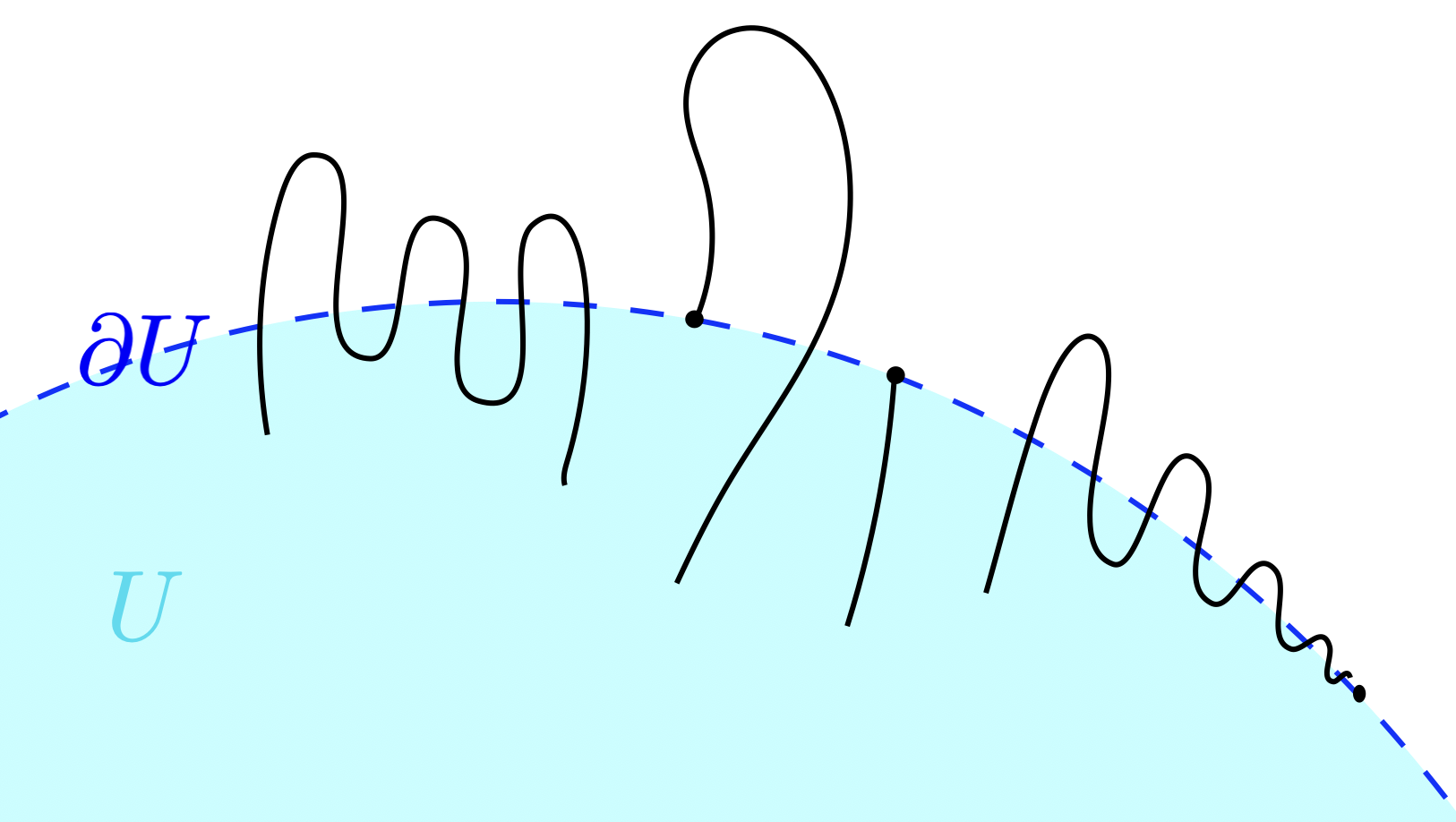}\hspace{1cm}\includegraphics[scale=0.245]{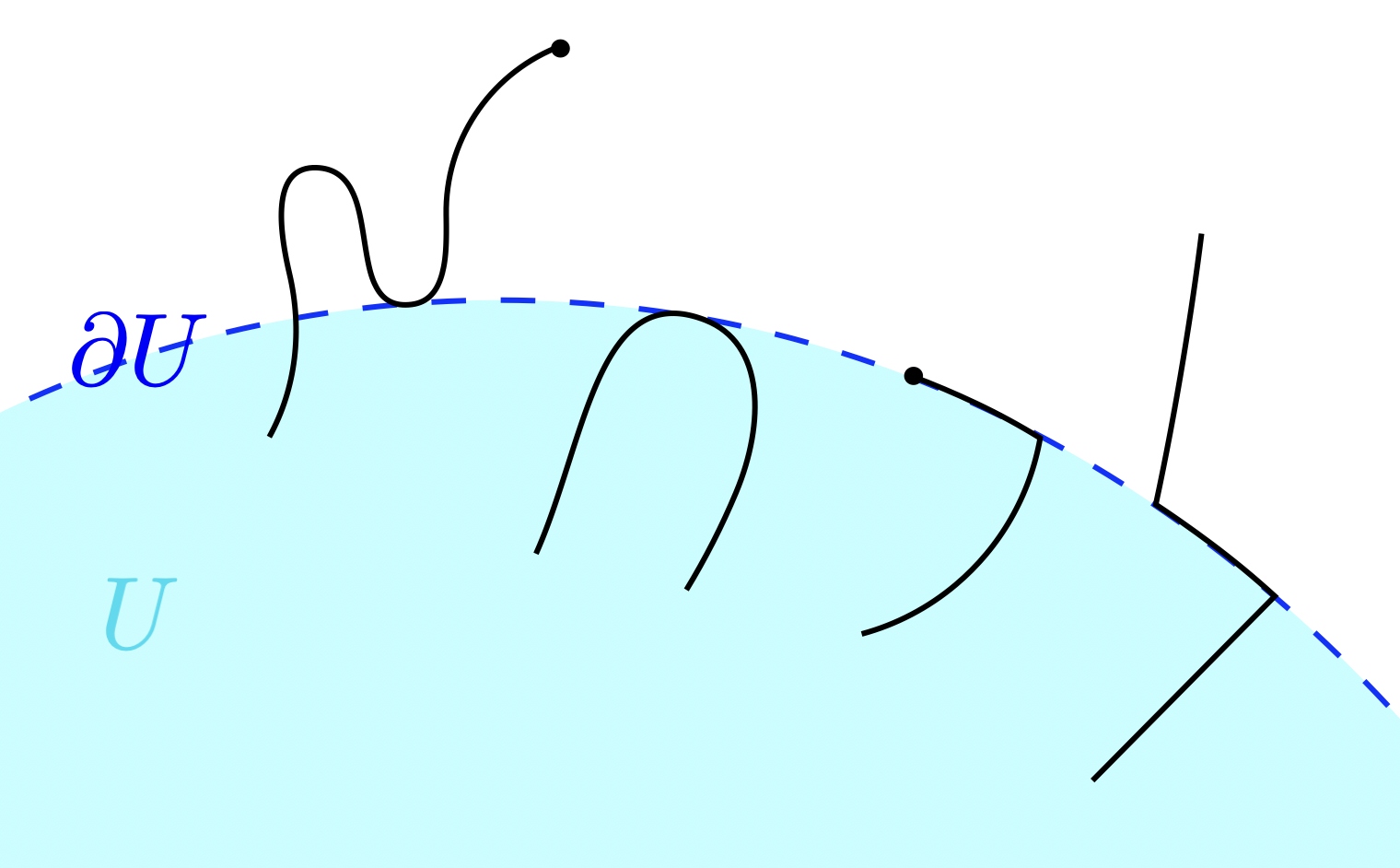}
\caption{A pierced open set (left),  versus an open set which is not pierced for four reasons (right).}\end{figure}

The space $X$ has the \textbf{Jure property} if $X$ has a basis of pierced open sets.  Subspaces of the Cantor fan clearly have the Jure property. These include all smooth fences and smooth fans.  The question was raised in \cite{pans}  as to whether every fan is Jure \cite[Problem 4.30]{pans}. In this paper we provide a positive answer.
 
\begin{ut}\label{t1}Every  fence has the Jure property.\end{ut}

\begin{uc}\label{t2}Every  fan has the Jure property.\end{uc}

\section{Preliminaries}

\subsection{Plane embedding}
The first thing we will show is that  every fence can be embedded into the plane. Standard facts about plane topology can then be used to prove Theorem 2. 

The proof below uses  Bing's method of embedding chainable continua into the plane \cite{bin1,bin2}, and is essentially given in \cite[Section 2]{ov1}. An \textbf{$\varepsilon$-chain} is a finite sequence $U_1,\ldots U_n$ of open sets of diameter $<\varepsilon$, such that $\overline{U_i}\cap \overline{U_j}\neq\varnothing$  if and only if $|i-j|<2$.  We call a cover $\mathcal C$ of $X$ an \textbf{$\varepsilon$-cover} if $\mathcal C$ is a finite union of disjoint $\varepsilon$-chains, and non-adjacent elements of $\mathcal C$ are a positive distance apart.

\begin{ut}\label{emb}Every fence is embeddable into $\mathbb R^2$. \end{ut}

\begin{proof}Let $X$ be a fence.  The decomposition of $X$ into components is a zero-dimensional compact metric space $K$, which can be viewed as a linearly ordered subspace of the real line $\mathbb R$ with $\max(K)=1$. Let $g:X\to K$ be the continuous mapping whose point pre-images are components of $X$. Let $$x_0=\sup \big\{x\in K:\text{ for each }x'\leq x\text{  there exists a } \textstyle\frac{1}{2}\text{-cover of }g^{-1}([0,x'])\big\}.$$

Assume that $g^{-1}\{x_0\}$ is an arc. Then  $g^{-1}\{x_0\}$ is covered by a chain of open intervals $U_1,\ldots,U_n$ whose diameters  are less than $\frac{1}{3}$ in $X$. For each $U_i$ let $V_i$ be open in $X$ such that $V_i\cap g^{-1}\{x_0\}=U_i$. Since $K$ is 0-dimensional and $X$ is compact, there exists a clopen set $H\subset K$ such that $g^{-1}(H)\subset \bigcup _{i=1}^n V_i$ and the sets $g^{-1}(H)\cap V_i$ form a $\frac{1}{2}$-chain. Moreover, we may assume that $H= [a,b]\cap K$ for two real numbers $a,b\in K$. Let $x_1$ be the maximum element of $K$ below $a$. By the definition of $x_0$, there exists a $\frac{1}{2}$-cover $\mathcal C$ of $g^{-1}([0,x_1])$. Since $g^{-1}([0,x_1])$ is open, we may assume that $\bigcup \mathcal C\subset g^{-1}([0,x_1])$. Then $\mathcal C\cup \{V_1,\ldots, V_n\}$ is an $\varepsilon$-cover of $g^{-1}[0,b]$. Thus $x_0=b$. If $b\neq 1$ then let $x_2$ be the next element of $K$ after $b$. Repeating the argument shows that $x_2$ is contained in a $\frac{1}{2}$-cover of $g^{-1}([0,x_2])$, which contradicts the definition of $x_0$. Thus $x_0=1$ and  $X$ has a $\frac{1}{2}$-cover $\mathcal C^1$.

There exists $\varepsilon>0$ such that each pair of  non-intersecting elements of $\mathcal C^1$ are $\varepsilon$ apart. By the Lebesgue number lemma \cite[Lemma 27.5]{mun}, there exists $\delta>0$ such that any set of diameter less than $\delta$ is contained in an element of $\mathcal C^1$.   Let $\gamma=\min\{\frac{1}{4},\frac{\varepsilon}{8},\delta\}$.   By the arguments above,  there exists a $\gamma$-cover $\mathcal C^2$ of $X$. In this manner, we obtain a sequence of $(\frac{1}{2})^k$-covers  $\mathcal C^k$ of $X$,  that are formed from finitely-many chains with disjoint closures, which refine the chains in  $\mathcal C^{k-1}$, and such that no chain in  $\mathcal C^k$ of less than nine links intersects two non-adjacent links in $\mathcal C^{k-1}$. 

By \cite[p.654]{bin1}, each  $\mathcal C^k$ can be represented in $\mathbb R^2$ as  a collection  $\mathcal D^k$ of chains consisting of interiors of rectangles, which follow the same patterns as chains in $\mathcal C^k$ with respect to $\mathcal C^{k-1}$.   By   \cite[Theorem 11]{bin2},  $(\bigcup \mathcal D^1)\cap (\bigcup \mathcal D^2)\cap \ldots$ is homeomorphic to $X$.\end{proof}

\begin{uc}\label{coo}Let $X$ be a fan with vertex $v$. If $V$ is any open subset of $X$ containing $v$, then $X\setminus V$ is embeddable into $\mathbb R^2$.\end{uc}

\subsection{Notation for radial sets in the plane}
For every $p\in \mathbb R^2$ and $r,r'\in (0,\infty)$,  let 
\begin{itemize}
\item $S_r(p)$ denote the circle in $\mathbb R^2$ of radius $r$ centered at $p$,
\item  $U_r(p)$ be the open $r$-ball centered at $p$, 
\item $D_r(p)$ be the  closed $r$-disc centered at $p$,  and
\item $A_{rr'}(p)$ be the closed annulus between  $S_r(p)$ and $S_{r'}(p)$. 
\end{itemize}  When the point $p$ is  fixed, we will denote these sets by simply   $S_r$, $U_r$, $D_r$,  $A_{rr'}$.

\subsection{Topological dimension}

A topological space $X$ is \textbf{0-dimensional} if $X$ has a basis of (cl)open sets with empty boundaries,  \textbf{1-dimensional} if $X$ is not $0$-dimensional but has a basis of open sets with $0$-dimensional boundaries, etc.

Most circles intersect   $1$-dimensional plane compacta at at $0$-dimensional sets. 

\begin{up}\label{bctz}Let $X\subset \mathbb R^2$ be compact and $1$-dimensional. Then for every $p\in X$, $$Z=\{r\in (0,\infty):S_r(p)\cap X \text{ is 0-dimensional}\}$$ is a dense $G_{\delta}$-set.\end{up}

\begin{proof}Let $\{\vartheta_j:j<\omega\}$ be an enumeration of the countable set $\mathbb Q\cap [0,2\pi]$. For all $i,j$ such that $\vartheta_i<\vartheta_j$,  define $\alpha_{rij}$ to be the arc of $S_r$ that goes from the point at $\vartheta_i$ radians to the point at $\vartheta_j$ radians (counterclockwise). Let $F_{ij}=\{r\in  (0,\infty):\alpha_{rij}\subset X\}$. By compactness of $X$, each $F_{ij}$ is closed. And since $X$ is $1$-dimensional, $F_{ij}$ cannot contain any interval; otherwise $X$ would contain a filled region of the plane in the form of a sector of an annulus. Therefore $F_{ij}$ is nowhere dense and $Z=(0,\infty)\setminus \bigcup_{i,j}F_{ij}$ is  dense $G_{\delta}$.
 \end{proof}
\subsection{Bends}Let $X\subset \mathbb R^2$ and  $p\in X$. For any arc $\beta$ let  $a_\beta$ and $b_{\beta}$ denote the two endpoints of $\beta$.  If $q,r\in (0,\infty)$,  we say that an arc $\beta\subset X$ is a \textbf{bend at  $S_r$ terminating at $S_q$}  if:
\begin{itemize}
\item $\beta\subset A_{rq}$, 
\item  $\beta\cap S_r\neq\varnothing$; and \item $\beta\cap S_{q}=\{a_\beta,b_\beta\}$.
\end{itemize}
Define an arc $\beta\subset X$ to be a \textbf{bend at }$S_r$ if there exists $q$ such that $\beta$ is a bend at $S_r$ terminating at $S_q$.

\begin{figure}\centering\includegraphics[scale=0.4]{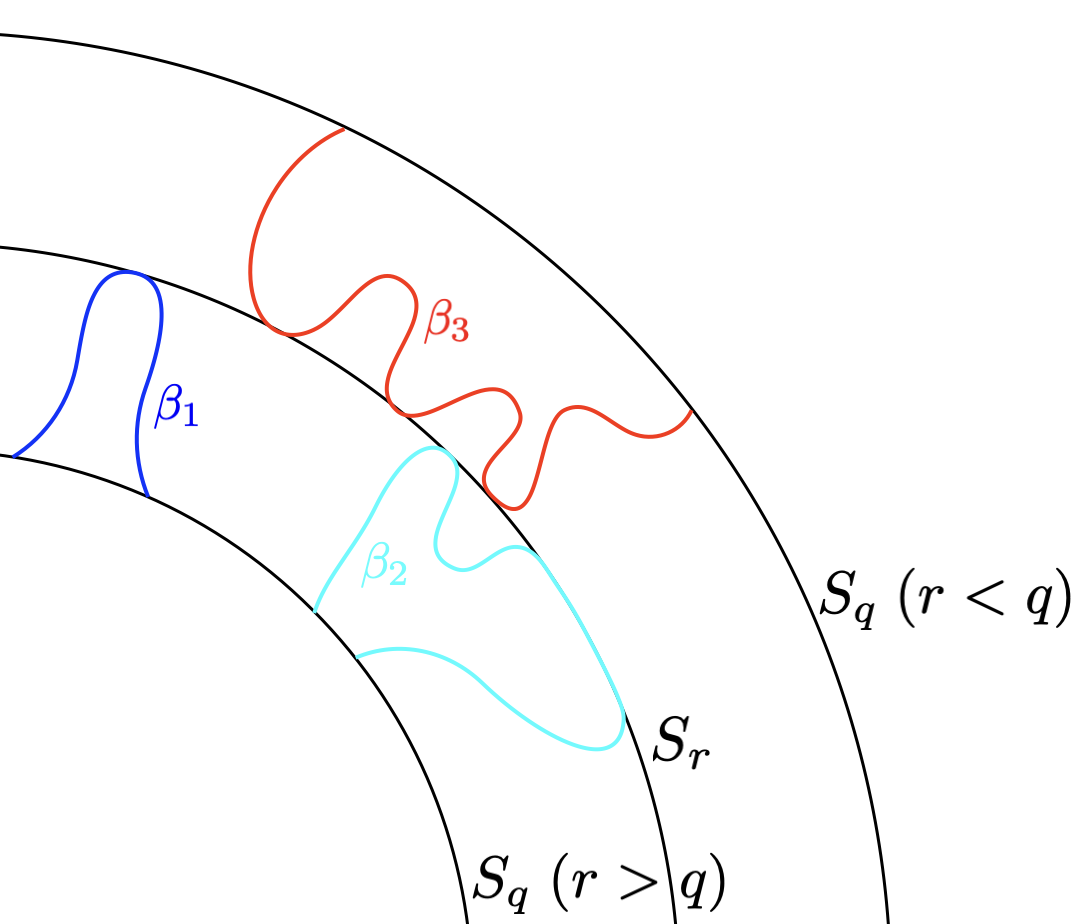}\caption{Bends at $S_r$ that terminate at $S_q$}\end{figure}

\begin{up}\label{bp} If $S_r\cap X$ is 0-dimensional and $X$ has no bend at $S_r$, then  $U_r\cap X$ is pierced. \end{up}

\begin{proof}Suppose that  $S_r\cap X$ is 0-dimensional and $x\in \partial (U_r\cap X)$ is not a piercing point of $U_r\cap X$.  Let $\alpha$ and $V$ witness this fact; so $x$ is a cut point of $\alpha$, and $\alpha\cap V$ is a subset of either $D_r$ or $\mathbb R^2\setminus U_r$. Let us  assume $\alpha\cap V\subset D_r$. Since $S_r\cap X$ is 0-dimensional, each component of $\alpha\setminus \{x\}$ must intersect $U_r$. Let $a,b\in U_r$ be points in different components of $\alpha\setminus \{x\}$. Let $\varepsilon>0$ be less than $\min\{d(a,x),d(b,x)\}$. Let $q=r-\varepsilon$. Let $\beta$ be the component of $x$ in $A_{qr}$. Then $\beta$ is a bend at $S_r$ terminating at $S_q$. \end{proof}

If $\beta$ is a bend at $S_r$ terminating at $S_q$, we say that $\beta$  \textbf{separates $\vartheta$ from }$\infty$ if there is an arc $\alpha\subset D_q$ such that $\alpha\cap S_q=\{a_\beta,b_\beta\}$, and  $\vartheta$ is contained in the   bounded component of  $\mathbb R^2\setminus (\alpha \cup \beta)$. Let $$\mathfrak B_{qr\vartheta}$$
be the set of all bends at $S_r$ that terminate at $S_q$ and separate $\vartheta$ from $\infty$. 

\subsection{Limsup of continua}
Given a sequence $A_1,A_2,\ldots$ of sets in a space $X$, let  
$\limsup A_n$ denote the set of points $x\in X$ such that if $U$ is open in $X$ and $x\in U$, then  $U \cap A_n\neq\varnothing$ for infinitely many $n$. 

\begin{up}[{\cite[Lemma 4.6]{pans}}] Let $X$ be a compact metric space, let $x\in X$, and for each $n$ let $A_n$ be a continuum in $X$ and let $x_n\in A_n$. If $x_n$ converges in $X$,  then $\limsup A_n$ is a continuum.\label{lsup}\end{up}

\section{Proof of Theorem 2}

Let $X$ be a fence. By Theorem \ref{emb} we may assume $X\subset \mathbb R^2$.

\begin{ul} For any given $p\in X$, the set 
$$E=\{r\in (0,\infty):X\text{ has a bend at }S_r(p)\}$$ is first category in $(0,\infty)$.\end{ul}

\noindent Combined with Propositions \ref{bctz} and \ref{bp}, this will show that   $X$ has the Jure property.

Toward proving Lemma 9, fix $p\in X$. Let $\{q_i:i<\omega\}$ and $\{\vartheta_j:j<\omega\}$ be enumerations of the countable sets  $\mathbb Q\cap (0,\infty)$ and  $\mathbb Q\cap [0,2\pi]$, respectively. Let $\vartheta^i_j$ be the point on the circle $S_{q_i}$ at the  angle  $\vartheta_j$ radians. Define $r\in F_{ij}$ if there exists a bend at $S_r$ terminating at $q_i$ and separating $\vartheta^i_j$ from $\infty$. That is, $F_{ij}$ is the set of all $r$'s such that  $\mathfrak B_{q_i r\vartheta_j}\neq\varnothing$.

\begin{ucl}  $$E=\bigcup_{i,j<\omega} F_{ij}.$$\end{ucl}  

\begin{proof}Clearly $E$ contains every $F_{ij}$. For the other inclusion, let $r\in E$. Let $\beta\subset A_{qr}$ be a bend at $S_r$ terminating at $S_q$. Choose $i$ so that $q_i$ is between $q$ and $r$. 
  Let $x\in \beta\cap S_r$, and let $\alpha$ be the component of $x$ in $\beta\cap A_{q_i r}\setminus S_{q_i}$.  Then $\beta'=\overline\alpha$ is a bend at $S_r$ terminating at $S_{q_i}$. Let $V$ be the bounded component of $\mathbb R^2\setminus (D_{q_i}\cup \beta')$. There exists $j$ such that  $\vartheta^i_j\in  \partial V\setminus \beta'$. Then $\beta'$ separates $\vartheta^i_j$ from $\infty$ and witnesses   $r\in F_{ij}$. \end{proof}

A countable sum of first category spaces is again first category. So to prove Lemma 9, it now suffices to show that every $F_{ij}$ is a first category subset of $(0,\infty)$. This will be accomplished with the following.

\begin{ul}Let $X$ be a fence in $\mathbb R^2$. For every $p\in X$, $q\in (0,\infty)$, and $\vartheta\in S_q(p)$,  $$F=\{r\in (0,\infty):\mathfrak B_{qr\vartheta}\neq\varnothing\}$$
is first category in $(0,\infty)$.\label{kl}\end{ul}

\subsection{Proof of Lemma 11}
Fix $p\in X$, $q\in (0,\infty)$ and $\vartheta\in S_q$. To prove that $F$ is first category, it suffices to show that $F\cap (0,q)$ and $F\cap (q,\infty)$  are each first category. We  will proceed with the latter, and note that symmetric arguments would apply to  $F\cap(0,q)$.

Since $X$ is non-separating, there is a ray $\gamma\subset \mathbb R^2\setminus X$ that goes from a single point of $S_q$ to $\infty$.   There is a homeomorphism $H:\mathbb R^2\setminus (U_q\cup \gamma)\to \mathbb H$ onto the closed half-plane $\mathbb H=\{\langle x,y\rangle\in \mathbb R^2:y\geq 0\}$ such that $H[S_q\setminus \gamma]= \mathbb R\times \{0\}$ and $H(\vartheta)=0$. This    establishes a linear ordering of $S_q\cap X$ by real numbers on the $x$-axis. For any two points $a,b\in S_q\cap X$, we will say  $a\prec b$ if $H(a)<H(b)$.  By re-labeling the endpoints of $\beta$ we can assume $a_\beta \prec b_\beta$ for all $\beta\in \mathfrak B_{qr\vartheta}$. 

Observe that $\beta$ separates $\vartheta$ from $\infty $ if and only if $a_\beta \prec \vartheta\prec b_\beta$.

\begin{ucl}[Bends at different radii, ordering of endpoints]\label{disb}Let $r,r'\in (q,\infty)$. If  $r<r'$, $\beta\in \mathfrak B_{qr\vartheta}$ and $\beta'\in \mathfrak B_{qr'\vartheta}$, then 
\begin{itemize}
\item[(i)] $\beta\cap \beta'\setminus S_q=\varnothing$; and
\item[(ii)]  $a_\beta'\preceq a_{\beta}$ and $b_\beta\preceq b_{\beta'}$.
\end{itemize}\end{ucl}

\begin{proof}For a contradiction, suppose that $\beta$ and $\beta'$ intersect outside of $S_q$. Let $x\in \beta'\cap S_{r'}$ and $y\in \beta\cap \beta'\setminus S_q$. Note that $x\notin \beta$. Let $\alpha\subset\beta\cup \beta'$ be an arc with endpoints $x$ and $y$. Let $z$ be the first point of $\alpha$ that belongs to $\beta$, as $\alpha$ is traversed going from $x$ to   $y$. Let $\xi\subset \alpha$ be the arc from $x$ to $z$. Then $\beta\cup \xi$ is a triod in $X$, which is a contradiction. This proves (i). As for (ii), note that  $x$ (and in fact all of $S_{r'}$) lies in the unbounded component of $\mathbb R^2\setminus (D_q\cup \beta)$. Since $\beta$ and $\beta'$ each separate $\vartheta$ from $\infty$, and $x\in \beta '$, by (i) it must connect $x$ to $S_q$ at $a_{\beta'}\preceq a_{\beta}$ and $b_{\beta'}\succeq b_\beta$. \end{proof}

\begin{ucl}[Upper limit of bends]Let $r_1,r_2,\ldots$ be an increasing sequence in $(q,\infty)$. If $r_n\to r$ 
  and $\mathfrak B_{qr_n\vartheta}\neq\varnothing$ for each $n$, then  $\mathfrak B_{qr\vartheta}\neq\varnothing$.\label{ulim}\end{ucl}

\begin{proof}For each $n$ let $\beta_n\in \mathfrak B_{qr_n\vartheta}$. Write $a_{\beta_n}$ and $b_{\beta_n}$ simply as $a_n$ and $b_n$. By Claim \ref{disb}(ii), along $S_q$ the sequence $a_n$ is decreasing and $b_n$ is increasing. 
   By compactness of $X$ these sequences are bounded away from $\gamma$ along $S_q$. Thus $a_n$ and $b_n$ converge to two points $a$ and $b$ in $S_q\cap X$ which satisfy  $a\preceq a_n$ and $b_n\preceq b$ for all $n$.  Let $\kappa=\limsup \beta_n$. Then $a,b\in \kappa\subset A_{qr}$, and by  Proposition \ref{lsup}  $\kappa$ is a continuum in $X$  (hence an arc).  Let $h:[0,1]\to \kappa$ be a homeomorphism such that $h(0)=a$ and $h(1)=b$.  Let 
\begin{align*}t_0&=\sup\{t\in [0,1]:h(t)\in S_q\text{ and } h(t)\preceq a\}\text{, and}\\ 
t_1&=\inf\{t\in [0,1]:h(t)\in S_q\text{ and } h(t)\succeq b\}.\end{align*}
 Then  $\beta=h([t_0,t_1])$ intersects $S_q$ at only two points $c=h(t_0)$ and $d=h(t_1)$. We have  $c\preceq a\preceq a_n$ and $b_n\preceq b\preceq d$ for all $n$. Thus $\beta$ must intersect $S_r$ because it has to go around every $\beta_n$ (into $A_{r_nr}$). Further,  $c\preceq a_0$ and $b_0\preceq d$ shows that $\beta$ separates $\vartheta$ from $\infty$. We conclude that $\beta\in \mathfrak B_{qr\vartheta}$.\end{proof}

\begin{ucl}[Lower limit of bends] Let $r_1,r_2,\ldots$ be a decreasing sequence in $(q,\infty)$.  If $r_n\to r$, $\beta_n\in \mathfrak B_{qr_n\vartheta}$ for each $n$, and $\beta\in \mathfrak B_{qr\vartheta}$, then  $\beta\cap \limsup \beta_n \neq\varnothing$. 
\label{llim}\end{ucl}

\begin{proof}By Claim \ref{disb}, $  a_n\preceq a_{n+1}\preceq a_\beta$ and $b_\beta\preceq b_{n+1}\preceq b_n$. Thus $a_n$ and $b_n$ converge in $S_q\cap X$ to two points $a$ and $b$ with  $a\preceq a_\beta$ and $ b_\beta\preceq b$. 
 By Proposition \ref{lsup} and convergence of $a_n$,  $\kappa=\limsup \beta_n$ is a continuum in $X$. Clearly $a,b\in \alpha$  and  $\kappa\subset A_{qr}$.   Since  $a\preceq a_\beta$ and $ b_\beta\preceq b$, we have that   $\beta\cup \gamma$ separates $a$ from $b$ in  the annulus $A_{qr}$. Hence  $ \beta\cap \kappa\neq\varnothing$.\end{proof}

\begin{ucl}Let $I$ be any interval of $(q,\infty)$. Then  $F\cap I$ is countable or  not dense in $I$.  \end{ucl}
\begin{proof}Let  $$Y=\bigcup \bigcup_{r\in F\cap I} \mathfrak B_{qr\vartheta}.$$

\textit{Case 1:} $Y$ is contained in  a single component of $C$ of $X$. Then $C$ is an arc. The interiors  of bends at different radii are disjoint by Claim \ref{disb}, and $C$ contains at most countably many pairwise disjoint  intervals.  Hence  $F\cap I$ is countable.

\textit{Case 2:} $Y$ is not contained in a single component of $X$. Then there exist $r<r'\in F\cap I$ and bends $\beta\in \mathfrak B_{qr\vartheta}$ and $\beta'\in B_{qr'\vartheta}$  that belong to  different components of $X$. By \cite[Lemma 1.4.4]{engd}, $X$ can be written as a union of two disjoint closed sets $A$ and $B$ such that $\beta\subset A$ and $\beta'\subset B$. Let $$a=\sup\{r\in (0,r']:A \text{ contains a member of }  \mathfrak B_{qr\vartheta}\}.$$  By Claim \ref{ulim}, there exists $\beta_a\in B_{qa\vartheta}$.  Let $$b=\inf \{r\in[a,r']:B \text{ contains a member of }  \mathfrak B_{qr\vartheta}\}.$$ 
By Claim \ref{llim},  $a=b$ would imply $A\cap B\neq\varnothing$.  So  $a\neq b$ and $(a,b)$ is a non-empty interval of $I$ missing $F$. Therefore $F$ is not dense in $I$.  \end{proof}

\begin{ucl}$F\cap (q,\infty)$ is a first category subspace of $(0,\infty)$. \end{ucl}

\begin{proof}Let $I_1,I_2,\dots$ be the non-degenerate components of $\overline{F\cap (q,\infty)}$.  Let $Q_n=F\cap I_n$ and $M=F\cap (q,\infty)\setminus \bigcup_{n=1}^\infty I_n$. By Claim 15 each $Q_n$ is countable, and clearly $M$ is nowhere dense in  $(0,\infty)$. Put $Q=\bigcup_{n=1}^\infty Q_n$. Then $F\cap (q,\infty)=Q\cup M$ is the union of a countable set and a nowhere dense set, which shows that it is first category.\end{proof}

The preceding can be obtained for $F\cap (0,q)$ with symmetric arguments. It follows that all of $F$ is first category. This completes the proof of Lemma 11. 

We are now ready for the main results. 

\subsection{The Jure property in fences}

\begin{ut}$X$ has the Jure property.\end{ut}

\begin{proof} Let $p\in X$ and let $U\subset X$ be any open set with $p\in U$. Let $\varepsilon>0$ such that $U_\varepsilon(p)\cap X\subset U$. Let $Z$ and $E$ be as defined in Proposition \ref{bctz} and Lemma 9. Proposition \ref{bctz} states that $Z$ is a dense $G_\delta$-set, and by Lemma 9  $(0,\infty)\setminus E$ contains a dense $G_{\delta}$-set. The intersection of two dense $G_{\delta}$-sets is dense, so there exists $r\in (0,\varepsilon)\cap Z\setminus E$. By Proposition \ref{bp}, $U_r\cap X$ is pierced. \end{proof}

This completes the proof of Theorem 2.

\section{The Jure property in fans}

\begin{uc}[Corollary 3]\label{pju}Every fan has the Jure property.\end{uc}

\begin{proof}Let $X$ be a fan with vertex $v$. Let $p\in X$, and let   $U$ be any open subset of $X$ containing $p$. Assume first that $p\neq v$. By shrinking $U$  we may assume  $v\notin \overline U$.  Applying Theorem 17 to the fence  $X\cap \overline U$ proves that $X$ has local basis of pierced open sets at $p$. In the case of $p=v$, let $V$ be an open set such that $p\in V\subset \overline V\subset U$.  By the previous case and compactness of $\partial V$, we may cover $\partial V$ with finitely many pierced open sets which are contained inside of $U$.  Their union with $V$ is a pierced open set in $X$  which contains $v$ and is contained in $U$. \end{proof}

\end{document}